\documentclass[bjps, english]{imsart}

\usepackage[utf8]{inputenc}
\usepackage[T1]{fontenc}
\usepackage{lmodern}
\usepackage{babel}

\usepackage{mathrsfs}
\usepackage{mathtools}
\usepackage{amsfonts}
\usepackage{amsthm}

\usepackage{graphicx}

\usepackage{booktabs}

\usepackage{hyperref}

\usepackage{mleftright} 
\mleftright

\usepackage[section]{placeins} 

\usepackage{accents} 

\usepackage[inline]{enumitem}

\usepackage{appendix}

\newtheorem{theorem}{Theorem}
\newtheorem{corollary}{Corollary}

\theoremstyle{remark}
  
\theoremstyle{definition}

\newcommand{\E}{\mathbb{E}}
\newcommand{\Prob}{\mathbb{P}}

\DeclareMathOperator*{\Cov}{Cov}

\usepackage{natbib}

\startlocaldefs
\numberwithin{equation}{section}
\theoremstyle{plain}

\endlocaldefs

\begin{document}

\begin{frontmatter}
\title{Mixing conditions of conjugate processes}
\runtitle{Mixing conditions of conjugate processes}

\begin{aug}
\author{\fnms{Eduardo} \snm{Horta}\thanksref{a}\ead[label=e1]{eduardo.horta@ufrgs.br}}
\and
\author{\fnms{Flavio} \snm{Ziegelmann}\thanksref{a}\ead[label=e2]{flavioaz@mat.ufrgs.br}}


\runauthor{Horta and Ziegelmann}

\affiliation[a]{Instituto de Matemática e Estatística -- Universidade Federal do Rio Grande do Sul}

\address{Address of the First and Second authors,\\
\printead{e1,e2}}

\end{aug}

\begin{abstract}
We give sufficient conditions ensuring that a $\psi$--mixing property holds for the sequence of empirical \textsc{cdf}s associated to a conjugate process.
\end{abstract}




\end{frontmatter}

\section{Introduction}
In Horta~and~Ziegelmann~\citeyearpar{horta2015conjugate} a \emph{conjugate process} is defined to be a pair $(\xi,\,X)$, where $X \coloneqq (X_\tau:\,\tau\geq 0)$ is a real valued, continuous time stochastic process, and $\xi \coloneqq (\xi_t:\, t=0,1,\dots)$ is a strictly stationary sequence of $M_1(\mathbb{R})$-valued\footnote{Here $M_1\left(\mathbb{R}\right)$ denotes the set of Borel probability measures on $\mathbb{R}$.} random elements, for which the following condition holds:
\begin{equation}\label{eq:conjugation}
	\Prob\left(X_\tau \in B\,\vert\,\xi_0,\xi_1,\dots\right) = \xi_t\left(B\right),\qquad \tau\in \left[t,t+1\right),
\end{equation}
for each $t = 0, 1, \dots$ and each Borel set $B$ in the real line. From the statistical viewpoint, the sequence $\xi$ is to be understood as a latent (i.e. unobservable) process, and thus all inference must be carried using information attainable from the continuous time, observable process $X$ alone. A crucial objective in this context is estimation of the operator $R^\mu\colon L^2\left(\mu\right)\rightarrow L^2\left(\mu\right)$ defined by
\begin{equation}
	R^\mu f\left(x\right) \coloneqq \int R_\mu\left(x, y\right) f\left(y\right)\,\mu\left(\mathrm{d}y\right),\qquad x\in\mathbb{R}
\end{equation}
where the kernel $R_\mu$ is given by
\begin{equation*}
	R_\mu\left(x,y\right) \coloneqq \int \Cov\left(F_0\left(x\right),F_1\left(z\right)\right)\Cov\left(F_0\left(y\right),F_1\left(z\right)\right)\,\mu\left(\mathrm{d}z\right),\qquad x,y\in\mathbb{R},
\end{equation*}
and where $\mu$ is a fixed, arbitrary probability measure on $\mathbb{R}$ equivalent to Lebesgue measure. In the above, $F_t\left(x\right) \coloneqq \xi_t\left(-\infty,x\right]$, $x\in\mathbb{R}$, is the (random) \textsc{cdf} corresponding to $\xi_t$.

One of the key results in Horta~and~Ziegelmann~\citeyearpar{horta2015conjugate} is Theorem~\ref{thm:LLN-for-psihat} below, which provides sufficient conditions under which $R^\mu$ can be $\sqrt{n}$-consistently estimated. Before stating the theorem, we shall shortly introduce the estimator $\widehat{R}^\mu$ which is (as one should expect) a sample analogue of $R^\mu$. Consider, for each $t = 1,\dots, n$, a sample of observations $\left\{X_{i,t}:\,i=1,\dots,q_t\right\}$ of size $q_t$ from $\left(X_\tau:\,\tau\in\left[t,t+1\right)\right)$. Typically one has $X_{i,t} = X_{t + (i-1)/q_t}$. We then let $\widehat{F}_t$ denote the empirical \textsc{cdf} associated with the sample $X_{1,t},\dots,X_{q_t,t}$,
\begin{equation*}
	\widehat{F}_t(x) \coloneqq \frac{1}{q_t}\sum_{i=1}^{q_t}\mathbb{I}[X_{i,t}\leq x],\qquad x\in \mathbb{R}.
\end{equation*}
Notice that both $F_t$ and $\widehat{F}_t$ are random elements with values in the Hilbert space $L^2(\mu)$, and thus we find ourselves in a framework similar to Horta~and Ziegelmann~\citeyearpar{horta2016identifying}.

In this setting, $\widehat{R}^\mu$ is defined to be the operator acting on $L^2\left(\mu\right)$ with kernel
\begin{equation*}
	\widehat{R}_\mu\left(x,y\right) \coloneqq \int \widehat{C}_1\left(x,z\right)\widehat{C}_1\left(y,z\right)\mu\left(\mathrm{d}z\right),\qquad x,y\in\mathbb{R},
\end{equation*}
where $\widehat{C}_1$ is the sample lag-1 covariance function
\begin{equation*}
	\widehat{C}_1\left(x,y\right) \coloneqq \frac{1}{n-1}\sum_{t=1}^{n-1}\left(\widehat{F}_t\left(x\right) - \bar{F}_0\left(x\right)\right)\times\left(\widehat{F}_{t+1}\left(y\right) - \bar{F}_0\left(y\right)\right),\qquad x,y\in\mathbb{R},
\end{equation*}
with $\bar{F}_0 \coloneqq \left(1/n\right)\sum_{t=1}^{n}\widehat{F}_t$.

Last but not least, let $X^{(t)}$ denote the stochastic process $\left(X_{t+\tau}:\,\tau\in\left[0,1\right)\right)$, so that $X^{(0)}, X^{(1)},\dots,X^{(t)},\dots$ is a sequence of $\mathbb{R}^{[0,1)}$-valued random elements. We say that a conjugate process $\left(\xi,\,X\right)$ is \emph{cyclic independent} if, conditional on $\xi$, we have that $\left(X^{(t)}:\,t=0,1,\dots\right)$ is an independent sequence. This means that, for each $n$ and each $(n+1)$-tuple $\mathcal{C}_0,\dots,\mathcal{C}_{n}$ of measurable subsets of $\mathbb{R}^{[0,1)}$, it holds that
\begin{equation}\label{eq:cyclic-independence}
	\Prob(X^{(0)}\in\mathcal{C}_0,\dots,X^{(n)}\in\mathcal{C}_n\,|\,\xi) = \prod_{t=0}^{n}\Prob(X^{(t)}\in\mathcal{C}_t\,|\,\xi).
\end{equation}

We are now ready to state the consistency theorem.

\begin{theorem}[Horta and Ziegelmann~\citeyearpar{horta2015conjugate}]\label{thm:LLN-for-psihat}
Let $\left(\xi,X\right)$ be a cyclic--independent conjugate process, and let $\mu$ be a probability measure on $\mathbb{R}$ equivalent to Lebesgue measure. Assume that $\big(\widehat{F}_t\colon\,t=1,2,\dots\big)$ is a $\psi$--mixing sequence, with the mixing coefficients $\Psi\left(k\right)$ satisfying $\sum_{k=1}^\infty k\,\Psi^{1/2}\left(k\right) < \infty$. Then it holds that
\begin{enumerate}[label={(\textit{\roman*}})]
	\item $\Vert \widehat{R}^\mu - R^\mu\Vert_{HS} = O_\Prob\left(n^{-1/2}\right)$;
	\item $\sup_{j\in \mathbb{N}}\vert \widehat{\theta}_j - \theta_j\vert = O_\Prob\left(n^{-1/2}\right)$.
\end{enumerate}
If moreover the nonzero eigenvalues of $R^\mu$ are all distinct, then
\begin{enumerate}[label={(\textit{\roman*}}), start = 3]
	\item $\Vert\widehat{\psi}_j - \psi_j\Vert_{L^2(\mu)} = O_\Prob\big(n^{-1/2}\big)$, for each $j$ such that $\theta_j > 0$.
\end{enumerate}
\end{theorem}

In the above, $\Vert\cdot\Vert_{HS}$ denotes the Hilbert-Schmidt norm of an (suitable) operator acting on $L^2(\mu)$, $(\theta_j\colon\,j\in\mathbb{N})$ (resp. $(\widehat{\theta}_j\colon\,j\in\mathbb{N})$) denotes the non-increasing sequence of eigenvalues of $R^\mu$ (resp. $\widehat{R}^\mu$), with repetitions if any\footnote{Notice that there is some ambiguity in defining things in this manner; to ensure that everything is well defined, we adopt the convention that the sequence $(\theta_j)$ contains zeros if and only if $R^\mu$ is of finite rank. Thus if the range of $R^\mu$ is infinite dimensional and $0$ is one of its eigenvalues, it will not show up in the sequence $(\theta_j)$. On the other hand, $\widehat{R}^\mu$ is always of finite rank.}, and, for $j\in\mathbb{N}$, $\psi_j$ (resp. $\widehat{\psi}_j$) denotes the unique eigenfunction associated with $\theta_j$ (resp. $\widehat{\theta}_j$).

\section{Main result}
In what follows it will be convenient to assume that the latent process is indexed for $t\in\mathbb{Z}$ and that the continuous time, observable process $X$ is indexed for $\tau\in\mathbb{R}$. That is, we update our definitions so that $\xi \coloneqq (\xi_t:\,t\in\mathbb{Z})$ and $X \coloneqq (X_\tau:\,\tau\in\mathbb{R})$. Recall (see Bradley~\citeyearpar{bradley2005basic}) that a strictly stationary sequence $(Z_t: t\in \mathbb{Z})$ of random elements taking values in a measurable space $\mathcal Z$ is said to be \emph{$\psi$-mixing} if the $\psi$-mixing coefficient $\Psi_Z$ defined, for $k\in\mathbb N$, by
\begin{equation}\label{eq:psi-mixing-coef}
	\Psi_Z(k) \coloneqq \sup \left|1 - \frac{\Prob(A\cap B)}{\Prob(A)\Prob(B)}\right|
\end{equation}
is such that $\Psi_Z(k)\to 0$ as $k\to\infty$, where the supremum in \eqref{eq:psi-mixing-coef} ranges over all $A \in \sigma(Z_t\colon\,t\leq 0)$ and all $B\in \sigma(Z_t\colon\,t\geq k)$ for which $\Prob(A)\Prob(B)>0$.

The $\psi$--mixing condition in Theorem~\ref{thm:LLN-for-psihat} imposes restrictions on the sequence of empirical \textsc{cdf}s $\big(\widehat{F}_t\big)$ and thus constrains $(F_t)$ and $(X_\tau)$ jointly. One could argue that it is more natural to impose a $\psi$--mixing condition on the latent process $(\xi_t)$ instead, the issue being that it may be the case that a mixing property  of the latter sequence is not inherited by $(\widehat{F}_t)$. 
If a condition slightly stronger than cyclic--independence is imposed, however, then inheritance does hold. This is our main result.

\begin{theorem}\label{thm:mixing-inheritance}
{Let $\left(\xi,\,X\right)$ be a cyclic--independent conjugate process, and let $\mu$ be a probability measure on $\mathbb{R}$ equivalent to Lebesgue measure. Assume $\xi$ is $\psi$--mixing with mixing coefficient sequence $\Psi_\xi$. If, for each $t$, the conditional distribution of $X^{\left(t\right)}$ given $\xi$ depends only on $\xi_t$, in the sense that the equality
\begin{equation}\label{eq:strong-cyclic-independence}
	\Prob\big[X^{\left(t\right)}\in\mathcal{C}\,\big\vert\,\xi\big] = \Prob\big[X^{\left(t\right)}\in\mathcal{C}\,\big\vert\,\xi_t\big]
\end{equation}
holds for each measurable subset $\mathcal{C}$ of $\mathbb{R}^{\left[0,1\right)}$ and each $t$, then $\big(X^{\left(t\right)}\big)$ is $\psi$--mixing with mixing coefficient sequence $\Psi_{X}\leq \Psi_\xi$.}
\end{theorem}

\begin{corollary}\label{thm:mixing-inheritance-corol}
	In the conditions of Theorem~\ref{thm:mixing-inheritance}, if $\sum_{k=1}^\infty k \Psi_\xi(k)^{1/2} <\infty$, then the $\psi$--mixing assumption of Theorem~\ref{thm:LLN-for-psihat} holds.
\end{corollary}

\begin{proof}[Proof of Theorem~\ref{thm:mixing-inheritance}]
For $k\in \mathbb{N}$, let $T_1$ and $T_2$ be finite, nonempty subsets of $\{0, -1, -2,\dots\}$ and $\{k,k+1,k+2,\dots\}$ respectively, and set $T_0 \coloneqq T_1\cup T_2$. Let $\{\mathcal{C}_t$, $t\in T_0\}$ be a collection of measurable subsets of $\mathbb{R}^{\left[0,1\right)}$. By definition, $\sigma(X^{(t)}:\,t\leq 0)$ coincides with the $\sigma$-field generated by the class of sets of the form $\bigcap_{t\in T_1}[X^{(t)}\in\mathcal{C}_t]$ over all finite, nonempty $T_1\subset\{0,-1,-2,\dots\}$ and all collections $\{\mathcal{C}_t:\,t\in T_1\}$ of measurable subsets of $\mathbb{R}^{[0,1)}$, and similarly for $\sigma(X^{(t)}:\,t\geq k)$.

Notice that by equation~\eqref{eq:strong-cyclic-independence} and the Doob--Dynkin Lemma (see \cite[Lemma 1.13]{kallenberg2006foundations}) we have $\Prob\left[X^{\left(t\right)}\in \mathcal{C}_t\,\vert\,\xi\right] = g_t\circ \xi_t$, for some measurable function $g_t\colon M_1\left(\mathbb{R}\right)\rightarrow\mathbb{R}$. This fact, together with the cyclic--independence assumption, ensures that $$\Prob\left\{\bigcap_{t\in T_j}\left[X^{\left(t\right)}\in\mathcal{C}_t\right]\right\} = \E \left\{\Prob\left\{\bigcap_{t\in T_j}\left[X^{\left(t\right)}\in\mathcal{C}_t\right]\,\Big\vert\,\xi\right\}\right\} =
\E\left\{\prod_{t\in T_j}g_t\circ \xi_t\right\},$$ $j = 0, 1, 2$ (a similar computation yields strict stationarity of the process $(X^{(t)}:\,t\in\mathbb{Z})$). Thus, the quantity
\begin{align}\label{eq:mixing-coef-X}
	\left\vert 1 - \frac{\Prob\left\{\bigcap_{t\in T_0}\left[X^{\left(t\right)}\in \mathcal{C}_t\right]\right\}}{\Prob\left\{\bigcap_{t\in T_1}\left[X^{\left(t\right)}\in \mathcal{C}_t\right]\right\}\Prob\left\{\bigcap_{t\in T_2}\left[X^{\left(t\right)}\in \mathcal{C}_t\right]\right\}}\right\vert
\end{align}
is seen to be equal to
\begin{align}\label{eq:mixing-coef-xi}
	\left\vert 1 - \frac{\E\left\{\prod_{t\in T_0}g_t\circ \xi_t\right\}}{\E\left\{\prod_{t\in T_1}g_t\circ \xi_t\right\}\E\left\{\prod_{t\in T_2}g_t\circ \xi_t\right\}} \right\vert.
\end{align}
Substituting each $g_t$ in \eqref{eq:mixing-coef-xi} by an arbitrary measurable, bounded and positive $g_t'\colon M_1\left(\mathbb{R}\right)\rightarrow \mathbb{R}$, and taking the supremum over all collections $\left\{g_t':\,t\in T_0\right\}$ of such $g_t'$, and over all $T_0 = T_1\cup T_2$ as above, gives an upper bound to \eqref{eq:mixing-coef-X}. It is easily seen\footnote{By definition $\Psi_\xi\left(k\right)$ is obtained by taking the supremum over all collections of $g_t'$ which are indicator functions of measurable subsets of $M_1\left(\mathbb{R}\right)$.} that this supremum yields precisely $\Psi_\xi\left(k\right)$. This establishes that $\Psi_X\left(k\right)\leq\Psi_\xi\left(k\right)$ and completes the proof.
\end{proof}

\begin{proof}[Proof of Corollary~\ref{thm:mixing-inheritance-corol}]
	By definition (or using the Doob--Dynkin Lemma) we have that $\widehat{F}_t$ is of the form $\widehat{F}_t = g_t\circ X^{(t)}$ for some measurable $g_t\colon \mathbb{R}^{[0,1)}\to L^2(\mu)$. Since $\Prob(\widehat{F}_t \in B) = \Prob(X^{(t)}\in g_t^{-1}(B))$, it follows that the supremum in the LHS over all measurable subsets $B$ of $L^2(\mu)$ is bounded above by $\sup \Prob(X^{(t)}\in\mathcal{C})$, with $\mathcal{C}$ ranging over all measurable subsets of $\mathbb{R}^{[0,1)}$. An easy adaptation of this argument shows that the mixing coefficient sequence $\Psi_{\widehat{F}}$ is bounded above by $\Psi_{X}$.
\end{proof}

\section{Examples}
We refer the reader to Horta~and~Ziegelmann~\citeyearpar{horta2015conjugate} for an interesting application of the theory of conjugate processes to the problem of financial risk forecasting. Below we provide a simple example to illustrate the theory.

As discussed in Horta~and~Ziegelmann~\citeyearpar{horta2015conjugate}, the case where $(\xi_t)$ is an independent sequence is of no interest, since in this case $R^\mu$ is trivially the zero operator. Consider then an iid sequence $(\vartheta_t:\,t\in\mathbb{Z})$, where $\vartheta_t$ is uniformly distributed on $[0,1]$, and let $\eta_t$ be the random probability measure defined by (abusing a little on notation) $\eta_t(0) = \vartheta_t$ and $\eta_t(1) = 1 - \vartheta_t$. Setting $\xi_t\coloneqq (\eta_t + \eta_{t-1})/2$, we clearly obtain a $\psi$-mixing sequence which satisfies the summability condition of Theorem~\ref{thm:LLN-for-psihat}. Indeed, $(\xi_t)$ is $1$-dependent. A straightforward computation shows that $\Cov(F_0(x),F_1(y)) = 1/48$ for $x,y\in[0,1)$ and is identically zero otherwise, and therefore $R_\mu(x,y)$ is a positive constant for $x,y\in[0,1)$ which only depends on the chosen measure $\mu$.

Now, aside from the assumption that relation \eqref{eq:conjugation} holds, the nature of the process $(X_\tau:\,\tau\in\mathbb{R})$ is rather arbitrary. Below we simulate the case where, conditional on $\xi$, the process $(X_{t+\tau}:\,\tau\in[0,1))$ is a continuous time Markov chain on the state space $\{0,1\}$ with stationary distribution $(\xi_t(0), \xi_t(1))$. There is a free parameter in the construction, which is the mean holding time $1/q_0$ of state $0$. We set $q_0=10$. Thus, conditional on $\xi_t = \lambda_t$, the process $(X_{t+\tau}:\,\tau\in[0,1))$ is a Markov chain with initial distribution $(\lambda_t(0), \lambda_t(1))$ and generator
\begin{equation*}
	Q = \begin{pmatrix}
		-q_0 & q_0\\
		r_t & -r_t
	\end{pmatrix}
\end{equation*}
where $r_t \coloneqq q_0\lambda_t(0) / \lambda_t(1)$.

The conjugate process $(\xi, X)$ described above can be informally summarized as follows. At each day, the world finds itself in a (unobservable) \emph{state} which is characterized by a number lying in $[0,1]$. Within each day, given the state of the world, a system can find itself in two distinct (observable) \emph{regimes} (say, regime $\bar{0}$ and regime $\bar{1}$). This system switches between $\bar{0}$ and $\bar{1}$ according to a stationary, continuous time Markov chain, where the state of the world in that day represents the probability of the system being on regime $\bar{0}$ at any given point in time within that day. Figure~\ref{fig:example} displays a simulated sample path for the first 4 days of the process just described.

\begin{figure}[h]
	\centering
	\caption{A simulated sample path. Even days are colored in red; odd days in blue.}\label{fig:example}
	\includegraphics[scale = 1]{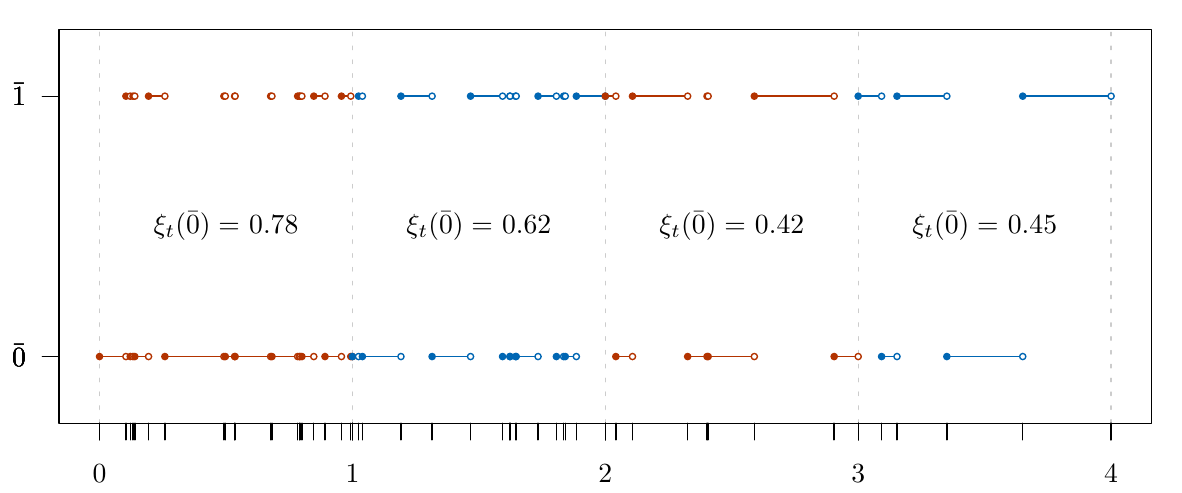}
\end{figure}

We also illustrate the consistency result via a Monte Carlo simulation study. For each $t = 1,\dots,n$, we sample the process $(X_{t+\tau}:\,\tau\in[0,1))$ once per cycle (that is, we take $q_t=1$ and $X_{1,t} = X_t$) and compute the corresponding value of $\widehat{C}_1(0,0)$. Figure~\ref{fig:boxplot} displays the boxplot of the estimated values of $\widehat{C}_1(0,0)$ across $10000$ replications of the above procedure, with the sample size varying in $\{100, 1000, 10000\}$. The blue line indicates the true parameter value $C_1(0,0) = 1/48$.
\begin{figure}[h]
	\centering
	\caption{Boxplots of $\widehat{C}_1(0,0)$ values across replications.}\label{fig:boxplot}
	\includegraphics[scale = 1]{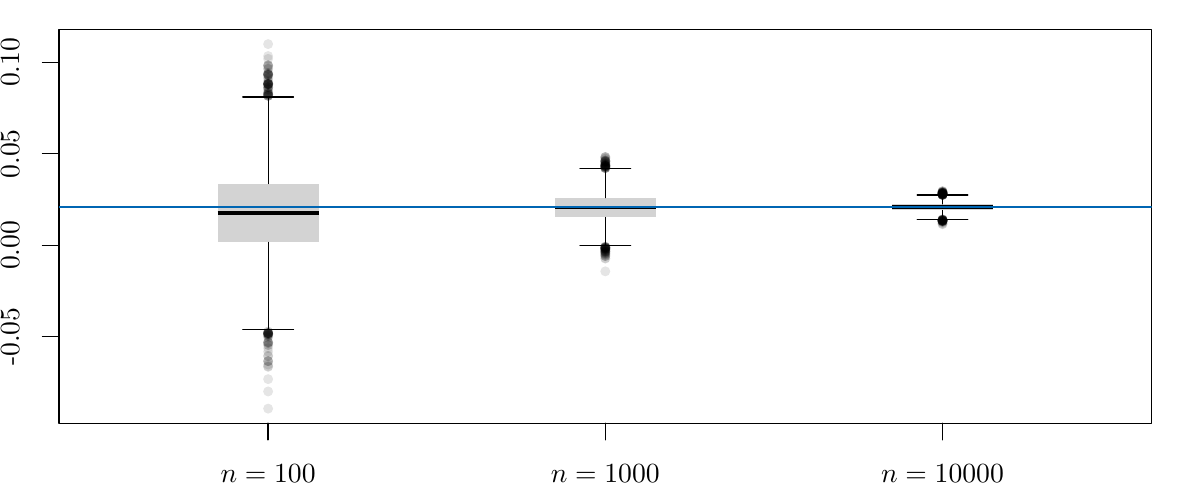}
\end{figure}

\end{document}